\documentclass[11pt]{amsart}
\usepackage{amsfonts}
\usepackage{graphicx}
\usepackage{amsmath}
\usepackage{amsthm}
\usepackage{amssymb}
\usepackage{amscd}
\usepackage{color}
\usepackage{url}

\usepackage{graphicx}
\usepackage{microtype}
\usepackage{enumitem}
\usepackage{pinlabel}
\usepackage[top=1.4in,bottom=1.4in,left=1.4in,right=1.4in]{geometry}

\theoremstyle{definition}

\newtheorem{theorem}{Theorem}[section]
\newtheorem{lemma}[theorem]{Lemma}
\newtheorem{corollary}[theorem]{Corollary}
\newtheorem{proposition}[theorem]{Proposition}
\newtheorem*{proposition*}{Proposition}
\newtheorem{question}[theorem]{Question}

\theoremstyle{definition}
\newtheorem{definition}[theorem]{Definition}
\newtheorem{remark}[theorem]{Remark}

\newtheorem{example}[theorem]{Example}

\newtheorem*{acknowledge}{Acknowledgements}

\newcommand{\R}{\mathbb{R}}

\newcommand{\N}{\mathbb{N}}
\newcommand{\Q}{\mathbb{Q}}

\newcommand{\id}{id}

\DeclareMathOperator{\supp}{supp}

\DeclareMathOperator{\Homeo}{Homeo}
\DeclareMathOperator{\mcg}{Map}

\newcounter{notes}
\renewcommand{\paragraph}[1]{\medskip 
\noindent \textbf{#1}}

\title{Automatic continuity for homeomorphism groups of noncompact manifolds}
\author{Kathryn Mann}
\date{}

\begin{document}

\begin{abstract}
We extend the proof of automatic continuity for homeomorphism groups of manifolds to non-compact manifolds and manifolds with marked points and their mapping class groups.  Specifically, we show that, for any manifold $M$ homeomorphic to the interior of a compact manifold, and a set $X \subset M$ homeomorphic to the union of a Cantor set and finite set, the {\em relative homeomorphism group} $\Homeo(M, X)$ and the {\em mapping class group}  $\Homeo(M, X)/\Homeo_0(M,X)$ have the property that any homomorphism from such a group to any separable topological group is necessarily continuous.
\end{abstract}

\maketitle

\section{Introduction} 

In \cite{MannAutCont} it was shown that the group of homeomorphisms of a compact manifold $M$ (equipped with the compact-open topology) satisfies {\em automatic continuity}: every homomorphism from this group to any separable, topological group is necessarily continuous.  The proof also applies to compact manifolds with boundary and to the subgroup $\Homeo(M, X)$ of homeomorphisms of any such manifold $M$ preserving a submanifold $X \subset M$ of dimension at least one.

Here we treat the remaining cases where $M$ is homeomorphic to the interior of a compact manifold, where $X$ has dimension 0, and also where $X$ is a Cantor set.  We also discuss automatic continuity for {\em mapping class groups} of infinite type surfaces, answering a question posed by N. Vlamis.  
Our motivation for the study of the invariant Cantor set case comes from the frequent appearance of groups of Cantor set-preserving homeomorphisms in dynamics.  For instance, 
Calegari \cite{Calegari} studies infinite groups acting on $\R^2$ with bounded orbits by collapsing each bounded component of the complement of an orbit closure to a point to produce an action of the group on a plane with an invariant finite or Cantor set, hence a homomorphism to $\Homeo(\R^2, X)$ where $X$ is finite or Cantor; this is exactly the kind of setting we have in mind.  

Our proofs use a condition formulated by Rosendal and Solecki in \cite{RS}.

\begin{definition}
A topological group $G$ is \emph{Steinhaus} if there is some $n \in \N$ such that, whenever $W \subset G$ is a symmetric set such that countably many left-translates of $W$ cover $G$, there exists a neighborhood of the identity of $G$ contained in $W^n$.  
\end{definition}  

As shown in \cite[Prop. 2]{RS}, the Steinhaus property implies automatic continuity (via a straightforward Baire category argument).  We show the following.

\begin{theorem} \label{thm:main} 
Let $M$ be a manifold, either compact or homeomorphic to the interior of a compact manifold with boundary.  Let $X \subset M$ be a set consisting of finitely many isolated points, or a Cantor set, or both. 
Then the group $\Homeo(M, X)$ of homeomorphisms preserving $X$ (setwise) is Steinhaus, hence has automatic continuity.  
\end{theorem} 

By contrast, when $M$ is noncompact (but is homeomorphic to the interior of a compact manifold $\bar{M}$), the group $\Homeo_c(M)$ of compactly supported diffeomorphisms (equipped with the compact-open topology) does not have automatic continuity, as the obvious inclusion of $\Homeo_c(M)$ into $\Homeo(\bar{M})$ is not continuous.  

Before embarking on the proof of Theorem \ref{thm:main}, we give an application to automatic continuity for some mapping class groups of infinite type surfaces, answering Vlamis' question and raising several new problems.

\section{Mapping class groups} 
The {\em mapping class group} of a surface $\Sigma$ is the group of homeomorphisms up to isotopy, equivalently the group $\pi_0(\Homeo(\Sigma))$.  When $\Sigma$ is {\em finite type} (compact or homeomorphic to the interior of a compact surface), these are discrete groups, and have been known since the work of Dehn to be finitely generated.  By contrast, mapping class groups of {\em infinite type} surfaces provide interesting examples of totally disconnected topological groups.  

We are interested more generally in understanding the algebraic and topological properties of 
 $\pi_0$ of homeomorphism groups, and of relative homeomorphism groups, of noncompact manifolds.  Such ``generalized mapping class groups'' are Polish groups: indeed, for any manifold $M$, and closed set $X \subset M$, the group $\Homeo(M,X)$ is a closed subgroup of the Polish group $\Homeo(M)$ (endowed with the compact-open topology), hence is Polish, and since $\Homeo_0(M,X)$ is a closed, normal subgroup of $\Homeo(M,X)$, the quotient $\Homeo(M, X)/\Homeo_0(M,X)$ is also Polish.   The context that originally prompted this note is the case where $M$ is a closed surface and $X \subset M$ homeomorphic to a closed subset of a Cantor set.  In this case, $\Homeo(M, X) \cong \Homeo(M - X)$ and $M-X$ is a finite genus topological surface (of infinite type, provided that $X$, its space of ends, is infinite).  
 \footnote{For this case where $M$ is a surface, that the mapping class group $\Homeo(M,X)/Homeo_0(M, X)$ is Polish was observed earlier in \cite{APV} using the property that these groups are the automorphism groups of a countable structure, namely the curve complex of the surface.}  
 
An easy argument (see \cite[Cor. 3]{RS}) shows that, provided that $G$ is a Polish, Steinhaus group, any Polish quotient group $H$ of $G$ is also Steinhaus.  Thus, we have the following consequence of Theorem \ref{thm:main}.

\begin{corollary}
Let $M$ be either compact or homeomorphic to the interior of a compact manifold with boundary, and $X \subset M$ the union of a Cantor set and finite set.  Then the {\em mapping class group} 
$\Homeo(M, X)/\Homeo_0(M,X)$ is Steinhaus, hence has the automatic continuity property.  
\end{corollary}

However, there are many examples of pairs where $\Homeo(M, X)$ fails to have automatic continuity, even in the case where $M = S^2$ and $X$ is a compact set.  
To see this, we recall first an example from \cite{RosendalAutCont}.  
\begin{example}[Example 1.4 of \cite{RosendalAutCont}] \label{ex:finite_index_nonopen}
Suppose $F$ is a finite group.  Let $\mathcal{U}$ be a nonprincipal ultrafilter on $\N$.  Then $H := \{(f_n) \in F^\N : \mathcal{U}_n(f_n) = 1\}$ is a finite index (of index equal to $|F|$) subgroup of the infinite product $F^\N$ that is not open.  Thus, the permutation action of $F^\N$ on cosets of $H$ gives a discontinuous homomorphism from $F^\N$ to the symmetric group $\mathrm{Sym}(F^\N/H)$.  
\end{example}

\begin{example} 
Let $X$ be a compact subset of $S^2$, homeomorphic to the disjoint union of a Cantor set $C$ and a countable set $Q$, with the properties that 
\begin{itemize}[label=$\bullet$, parsep=0pt, partopsep=2pt, itemsep=0pt]
\item The space of accumulation points of $Q$ in $C$ is homeomorphic to the one point compactification of $\N$.
\item For each $n \in \N$, $\overline{Q} \cap C$ has exactly $k$ points of rank $n$ (in the sense of Cantor-Bendixon rank).
\item For each $n \in \N$, the homeomorphism group of $X$ acts transitively on the $k$ points of rank $n$ in $\overline{Q} \cap C$.
\end{itemize}
Such a set is easily constructed, for instance one may take for each $n \in \N$ the union of a Cantor set and the set $\omega^n k +1$ with the order topology, glued along the k points of $\omega^n k +1$ of maximal rank; declare the $n$th such set to have diameter $2^{-n}$, and have them converge (in the Hausdorff sense) to a single point. 

Then $\Homeo(S^2, X)$ acts on $X$ by homeomorphisms, and the map $\Homeo(S^2, X) \to \Homeo(X)$ is surjective (see e.g. \cite{Richards} for a general proof that the homeomorphism group of a surface surjects to the space of homeomorphisms of its ends).  Considering only the invariant set $\overline{Q} \cap C$ gives a natural surjective homeomorphism $\Homeo(X) \to (\mathrm{Sym}_k)^\N$.   The map $\Homeo(S^2, X) \to (\mathrm{Sym}_k)^\N$ is open, 
so the composition with the homomorphism furnished by Example  \ref{ex:finite_index_nonopen} gives a discontinuous homomorphism from this group to a finite symmetric group.  
\end{example}

\begin{question} 
For which infinite type surfaces $\Sigma$ does $\mcg(\Sigma)$ have automatic continuity?   
\end{question} 
One might wish to treat separately the case where $\Sigma$ is finite genus and infinite genus, or the special case of finitely many ends but infinite genus.  

The argument for the non-examples above passes through the action of the homeomorphism group on the space of ends of the surface, suggesting two additional parallel questions. 

\begin{question}
Let $X$ be a closed subset of a Cantor set.  What topological conditions on $X$ ensure that $\Homeo(X)$ has automatic continuity?   Is some degree of local homogeneity required? 
\end{question}

\begin{question} 
Let $\Sigma$ be an infinite type surface.  Under what conditions on the topology of $\Sigma$ does the subgroup of homeomorphisms  fixing each end of $\Sigma$, and/or its identity component the {\em pure mapping class group} have automatic continuity? 
\end{question}

\section{General tools} 

This section contains some standard and some preliminary results to be used in the proof of Theorem \ref{thm:main}.  
We assume that manifolds are metrizable, but otherwise arbitrary.  For a topological manifold $M$ we take the standard compact-open topology on $\Homeo(M)$, which is separable and completely metrizable, i.e. Polish.   As recalled above in the introduction, if $X \subset M$ is a closed set, then $\Homeo(M, X)$ is a closed subgroup of $\Homeo(M)$, so also Polish, and in particular a Baire space.    
The following lemma is widely used in automatic continuity arguments.  
  
\begin{lemma} \label{W2dense}
Let $M$ be a manifold, $X \subset M$ a closed set, and $W \subset \Homeo(M,X)$ a symmetric set such that $\Homeo(M,X)$ is a countable union of left translates of $W$.  Then there exists a neighborhood $U$ of the identity in $\Homeo(M,X)$ such that $W^2$ is dense in $U$.
\end{lemma}

\begin{proof} 
Since $\Homeo(M,X)$ is a Baire space, and each left translated of $W$ is homeomorphic to $W$, it follows that the set $W$ is not meagre, so dense in some open set of  $\Homeo(M,X)$.  Since $W$ is symmetric, it follows that $W^2$ is dense in a neighborhood of the identity. 

\end{proof}
  
The following technical lemma generalizes  \cite[Lemma 3.8]{MannAutCont}, which is modeled after arguments of Rosendal from \cite{RosendalSurfaces}.    Recall that the {\em support} of a homeomorphism $f$ is the closure of the set $\{x  \mid f(x) \neq x\}$, and a group $G$  is said to have {\em commutator length} $p$ if each element of $G$ can be written as a product of $p$ commutators in $G$.  

\begin{lemma} \label{techlem}
Let $M$ be a manifold, and $W \subset \Homeo(M)$ a symmetric set such that $\Homeo(M)$ is a countable union of left translates of $W$.  Let $\mathcal{A}$ be a family of open subsets of $M$ satisfying: 
\begin{enumerate}
\item There exists an infinite family of pairwise disjoint, closed sets $U_i \subset M$ such that each set $U_i$ contains an infinite family of pairwise disjoint sets belonging to $\mathcal{A}$. 
\item There exists $p \in \N$ such that, for each $A \in \mathcal{A}$, the group of homeomorphisms with support on $A$ has commutator length bounded by $p$.  
\end{enumerate}
Then there exists $A \in \mathcal{A}$ such that each homeomorphism supported on $A$ is contained in $W^{8p}$.  
\end{lemma}

\begin{proof}

Let $G = \Homeo(M)$, and for any set $U \subset M$, let $G(U)$ denote the homeomorphisms with support contained in $U$.  Suppose that $G = \bigcup_{i=1}^\infty g_i W$, for some symmetric set $W$, and let $U_1, U_2, ...$ be the closed, disjoint sets from the statement of the lemma.  
\medskip 

\noindent \textit{Step 1. }  We first claim that there is some $U_i$ such that, for each $f \in G(U_i)$, there exists $w_f \in g_i W$, with support in the closure of $\bigcup_i U_i$, and such that the restriction of $w_f$ to $U_i$ agrees with $f$.
For if this statement does not hold, then there is a sequence of counterexamples $f_i \in G(U_i)$ such that each $f_i$ does not agree with the restriction to $U_i$ of any element of $g_i W$ supported on $\bigcup_i U_i$.  Define a homeomorphism $F(x)$ by 
$$F(x) = \left\{ \begin{array}{ll} f_i(x) & \text{ if } x \in U_i \text{ for some } i  \\
x & \text{ otherwise} 
 \end{array} \right.$$
By assumption, for some $i$ we have $F \in g_i W$, but $F$ restricts to $f_i$ on $U_i$.  This gives the desired contradiction. 

Now, given $f \in G(U_i)$, consider the homeomorphisms $w_{\id}$ and $w_{f}$ obtained above.  Then the restriction of $(w_{\id})^{-1} w_f$ to $U_i$ agrees with $f$ on $U_i$, and $(w_{\id})^{-1} w_f \in Wg_i^{-1} g_i W =  W^2$.  Thus, we conclude that for any $f \in G(U_i)$, there exists some element in $W^2$ agreeing with $f$ on $U_i$.  
\medskip

\noindent \textit{Step 2.} 
Let $U = U_i$ be the open set obtained from step 1, and let $A_1, A_2, ... $ be a countable disjoint family of open sets in $U$, with each $A_i \in \mathcal{A}$.   
We may apply the argument from Step 1 above to the family consisting of the closures of the $A_i$ in place of $U_1, U_2, ...$ to conclude that there exists $i$ such that for every $f \in G(A_i)$, there exists an element in $W^2$, with support contained in the closure of the union of the $A_i$, and agreeing with $f$ on $U$.   Forgetting subscripts, let $A = A_i$ denote this set. 

Let $f \in G(A)$.  Using the bounded commutator length assumption, write $f=[a_1, b_1]...[a_p, b_p]$, where $a$ and $b$ also have support in $A$.
Since in particular, each $a_j$ has support in $U$, there exists $w_{a_j} \in W^2$ with $\supp(w_{a_j}) \subset U$ and such that the restriction of $w_{a_j}$ to $U$ agrees with $a_j$.  There also exists $w_{b_j} \in W^2$ with $w_{b_j} \in G(A)$ and such that the restriction of $w_{b_j}$ to $A$ agrees with $b_j$.  Since $\supp(w_{a_j}) \cap \supp(w_{b_j}) \subset A$, we have $[a_j, b_j] = [w_{a_j}, w_{b_j}]$, hence 
$f = [w_{a_1}, w_{b_1}]...[w_{a_p}, w_{b_p}] \in W^{8p}.$

\end{proof}

The other general tool that we will use in the proof is the following.  

\begin{proposition} \label{prop_compact}
Let $M$ be a manifold, and suppose that $\Homeo(M) \subset g_i W$ for some symmetric set $W$.  Then there exists $n$, depending only on the dimension of $M$, and a neighborhood $U$ of the identity in $\Homeo(M)$ so that the following holds: 
If $K \subset M$ is any compact set, then any homeomorphism in $U$ with support contained in $K$ is an element of $W^n$.  
\end{proposition}
This follows directly from the proof of the main theorem of automatic continuity for compact manifolds in \cite{MannAutCont}.

\section{First case: $X$ finite or empty.} 

We begin with the case $X = \emptyset$; given the results of \cite{MannAutCont}, this reduces to the case where $M$ is homeomorphic to the interior of a compact manifold with boundary.  
A major technical ingredient for the proof is adapted from Rosendal--Solecki's proof of automatic continuity for the group of order-preserving automorphisms of $\Q$ in \cite{RS}.  
The argument here also covers the case where $X$ is finite, by the following remark. 

\begin{remark}[$X = \emptyset$ case implies $X$ finite case]   \label{rk_emptytofinite} 
Suppose $M$ is the interior of a compact manifold $\bar{M}$, and $X \subset M$ is a finite set.  
Since $\R^n - \{0\}$ is homeomorphic to $(0,1) \times S^{n-1}$, the manifold $M - X$ is homeomorphic to the interior of a compact manifold $\bar{N}$ with $\partial \bar{N}$ the union of $\partial \bar{M}$ and a disjoint union of spheres, one for each point of $X$.   Though homeomorphisms of $N$ do not necessarily extend to homeomorphisms of $\bar{N}$, they do extend to the space obtained by one-point-compactifying each spherical end, and this gives a topological isomorphism $\Homeo(N) \to \Homeo(M, X)$.   
\end{remark} 

\begin{proof}[Proof of automatic continuity for $M$ noncompact, $X = \emptyset$]
Let $M$ be homeomorphic to the interior of a compact manifold $\bar M$.    Let $W \subset \Homeo_0(M)$ be a symmetric set such that $\Homeo_0(M) = \bigcup_{i \in \N} g_i W$.  

Identify a neighborhood of the ends of $M$ with $([0, \infty) \times \partial \bar{M}) \subset M$.   Let $U$ be a neighborhood of the identity in $\Homeo_0(M)$ small enough so that, if $f \in U$, then $f$ can be written as $f= k \circ h$ where $k$ has support on the compact set $K = M - ((3, \infty) \times \partial M)$, and $h$ has support on $[2, \infty )\times \partial \bar{M}$.    
Proposition \ref{prop_compact} shows that, provided $U$ is chosen small enough, there exists $n$ depending only on the dimension of $M$ such that $k \in W^n$.   We wish to prove the same is true for $h$.   

Any homeomorphism $h$ of $[0, \infty )\times \partial M$ can be factored as $h_1 h_2$ where each $h_i$ has support on a set of the form $X_i \times \partial \bar{M}$, and $X_i \subset [0, \infty)$ is an infinite disjoint union of open intervals homeomorphic in $[0, \infty)$ to $\bigcup_{n=1}^\infty (2n, 2n+1)$.  (See e.g. \cite[Prop. 5.1]{Scottish}.)
Thus, it suffices to find $n$ such that $W^n$ contains any homeomorphism with support on a set of the form $X \times \partial \bar{M} \subset [0, \infty) \times \partial \bar{M}$ where $X$ is homeomorphic to $\bigcup_{n \in \N} (2n, 2n+1)$.     

Let $X$ be such a set.  Reparametrizing $\R$,  we may assume that in fact
$$X = \bigcup_{n \in \N} (2n, 2n+1) \times \partial \bar{M}.$$
We begin by applying Lemma \ref{techlem} with the class $\mathcal{A}$ consisting of sets of the form $A_\Lambda = Y_\Lambda \times \partial \bar{M}$, where $Y_\Lambda = \bigcup_{n \in \Lambda} (2n, 2n+1)$ for some infinite set $\Lambda \subset \N$.   Note that $\mathcal{A}$ satisfies the hypotheses of the lemma, since 
\begin{enumerate}
\item We may write $\N$ as a countable disjoint union of infinite sets $\Lambda_i$, and define $U_i$ to be the closure of $Y_{\Lambda_i} \times \partial \bar{M}$.  Each such set contains a countable union of disjoint elements of $\mathcal{A}$.   
\item Any element supported on such a set $A_\Lambda$ may be written as a single commutator.  
\end{enumerate} 
The proof of item ii) is a standard argument, we briefly recall this for completeness:  given 
$f$ supported on such a set $Y_\Lambda \times \partial \bar{M}$, let $I_n \subset (2n, 2n+1)$ be closed intervals such that $\supp(f) \subset \bigcup_{n\in \Lambda} I_n \times \partial M$.  Let $T: \R \to \R$ be a homeomorphism with support in $Y_\Lambda$ such that $T(I_n) \cap I_n = \emptyset$;  abusing notation, identify $T$ with the homeomorphism $T \times \id$ of $[0,\infty) \times \partial \bar{M}$. 
Let $a = \prod_{j\geq 0} T^jfT^{-j}$, which has support in  $\bigcup_{n \in \Lambda} I_n$.  Then $a \circ T a^{-1} T^{-1} = f$.  
\medskip

Thus we conclude (using the notation from Lemma \ref{techlem}) that for some such set $A_\Lambda \in \mathcal{A}$, we have $G(A_\Lambda) \subset W^8$.  Now we apply a trick used in \cite{RS}. For $\alpha \in \R$, let $\Lambda_\alpha$ be infinite subsets of $\Lambda$, such that $\Lambda_\alpha \cap \Lambda_\beta$ is finite for all $\alpha \neq \beta$.  Such a collection may be obtained, for example, by putting $\Lambda$ in bijective correspondence with $\Q$, and choosing $\Lambda_\alpha$ to be a sequence of distinct rational numbers converging to $\alpha$.

For each $\alpha$, let $f_\alpha \in \Homeo[0, \infty)$ satisfy 
$f_\alpha(2n) = 2n+1$, and $f_\alpha(2n+1) \in 2 \Lambda_\alpha$ for all $n \in  \Lambda_\alpha$, so that $f_\alpha$ maps each interval of $Y_{\Lambda_\alpha}$ to the interior of a connected component of $\R - Y_{\Lambda_\alpha}$.  Again abusing notation, identify these with the homeomorphisms $f_\alpha \times \id$ of $[0, \infty) \times \partial \bar{M}$.
Since $\R$ is uncountable, there is some $\alpha$ and some $\beta$ such that $f_\alpha$ and $f_\beta$ are in the same left-translate $g_i W$.  Thus, $f_\alpha^{-1} f_\beta$ and $f_\beta^{-1}f_\alpha$ are both in $W^2$.  
We will now use these two elements to conjugate (a suitable decomposition of) homeomorphisms supported on $X \times \partial \bar{M}$ into $A_\Lambda$, this will complete the proof.  

If $n \in \N - \Lambda_\alpha$, then $f(2n, 2n+1) \subset (2m, 2m+1)$ for some $m \in \Lambda_\alpha$.  If $m \notin \Lambda_\beta$, then $f_\beta^{-1}f_\alpha(2n, 2n+1)$ is contained in an interval of the form $(2k, 2k+1)$ where $k \in \Lambda_\beta$.  
Since $\Lambda_\alpha \cap \Lambda_\beta$ is finite, we conclude that, with the exception of finitely many values of $n$, the map $f_\beta^{-1}f_\alpha$ takes intervals of the form $(2n, 2n+1)$ where $n\notin \Lambda_\alpha$ into some interval in $Y_{\Lambda_\alpha} \subset Y_\Lambda$.  
Reversing the role of $\alpha$ and $\beta$, the same argument shows that, with only finitely many exceptions, $f_\alpha^{-1}f_\beta$ takes every interval of the form $(2n, 2n+1), n\notin \Lambda_\beta$, into some interval of $Y_\Lambda$.  
Let $F$ denote the union of these two exceptional sets of integers.  Thus we can decompose $X$ as the union of 
$X_1 : = \bigcup_{n \notin (\Lambda_\alpha  \cup F)}(2n, 2n+1) \times \bar{M}$, with $X_2 := \bigcup_{n \notin (\Lambda_\beta  \cup F)}(2n, 2n+1) \times \bar{M}$ and $X_3 := \bigcup_{n \in F}(2n, 2n+1) \times \bar{M}$, hence $G(X) = G(X_1)G(X_2)G(X_3)$.  

As we have shown above,  
$$f_\beta^{-1}f_\alpha G(X_1) (f_\beta^{-1}f_\alpha)^{-1} \subset G(Y_\Lambda) \subset W^8, \text{and similarly}$$
$$f_\alpha^{-1}f_\beta G(X_2) (f_\alpha^{-1}f_\beta^{-1}) \subset G(Y_\Lambda) \subset W^8.$$   
Thus $G(X_1) \subset W^{12}$ for $i=1,2$.  
Since $X_3$ is compact, we also have that $G(X_3) \subset W^n$ by Proposition \ref{prop_compact}.    It follows that $G(X) \subset W^{24+n}$, which proves the theorem.  

\end{proof} 


\section{General case}
  
Let $M$ be a topological manifold, either compact or homeomorphic to the interior of a compact manifold with boundary, and let $X \subset M$ be the union of a Cantor set and a (possibly empty) finite set.  Remark \ref{rk_emptytofinite} reduces the proof to the case where $X$ is a Cantor set, since if $X'$ is the set of isolated points of $X$, then $M - X'$ is homeomorphic to interior of a compact manifold with boundary, and $\Homeo(M, X) \cong \Homeo(M-X', X-X')$.  Thus, we assume going forward that $X$ is a Cantor set.  

\begin{proof}[Proof for $X \subset M$ a Cantor set]
Let $W \subset \Homeo(M, X)$ be a symmetric set such that $\Homeo(M, X) = \bigcup_i g_i W$.   For convenience, fix a metric on $M$.  Then a neighborhood basis of the identity on $M$ is given by sets of the form 
$$U_{K, \delta} := \{f \mid d(f(x), x)< \delta \text{ for all }x\in K\}$$
 as $K$ ranges over compact sets of $M$, and $\delta >0$.  

Call a closed ball in $M$ a {\em separating ball} if its boundary is disjoint from $X$ and separates $X$ into two components.  By Lemma \ref{W2dense}, $W^2$ is dense in a neighborhood of the identity of $\Homeo(M,X)$.  Thus, we may choose $\epsilon>0$ such that the following holds:

\begin{lemma}   \label{lem:expanding}
Let $x_1, x_2, ... x_n$ be any collection of points in $X$, with $d(x_i, x_j) > 2\epsilon$ for all $i \neq j$.  If $D_i$ and $E_i$ are both separating balls contained in the $\epsilon$-ball about $x_i$, then for any $\delta > 0$, there exists $f \in W^2$ such that, for all $i = 1, 2, ..., n$ the ball $f(D_i)$ is hausdorff distance at most $\delta$ from $E_i$.   
\end{lemma}

\noindent The proof is immediate, and the statement also holds if we replace the condition that $D_i$ and $E_i$ are separating balls with the condition that $\overline{D_i} \cap X = \emptyset$ and $\overline{E_i} \cap X = \emptyset$ holds for all $i$.

Now take two $2\epsilon$-separated sets, say $\{ x_1, x_2, ..., x_m\}$ and $\{x'_1, x'_2, ..., x'_{m'}\}$, as in Lemma \ref{lem:expanding}, and let $D_i \subset B_\epsilon(x_i)$ and $D'_i \subset B_\epsilon(x'_i)$ be separating balls such that
\begin{enumerate}
\item The union of the $D_i$ and the $D'_i$ cover $X$, and 
\item Each point of $X$ is contained in at most one element of $\{D_1, ...,  D_n, D'_1, ..., D'_{m'}\}$.  
\end{enumerate}

We now apply Lemma \ref{techlem} to sets consisting of separating balls.  Let $\mathcal{A}$ be the collection of sets satisfying the following condition: 
$A \in \mathcal{A}$ iff $A$ consists of a union of disjoint separating balls, with exactly one ball in each metric ball $B_\epsilon(x_i)$. 
Let $\mathcal{A'}$ be the analogous set for the metric balls $B_\epsilon(x_i')$.  It is easily checked that $\mathcal{A}$ and $\mathcal{A'}$ satisfy the conditions of Lemma \ref{techlem}, so we conclude that there exists $A \in \mathcal{A}$ and $A' \in \mathcal{A'}$ so that $G(A) \subset W^8$ and $G(A') \subset W^8$.  

Let $Z$ be a connected, simply connected neighborhood of $X$ such that the interior of $M-Z$ together with the union of the discs $D_i$ and $D'_i$ covers $M$.  By the fragmentation lemma (following Edwards--Kirby \cite{EK}), there exists a neighborhood $U$ of the identity in $M$ such that any $f \in U$ can be written as $g_0 g_1 g_2$ with $g_0$ supported on $Z$, $g_1$ supported on $\bigcup_i D_i$ and $g_2$ supported on $\bigcup_i D'_i$.  

Since each point of $X$ is contained in at most one such supporting set, we additionally have that if $f(X) = X$, then each $g_i$ must preserve $X$ as well.   Repeating the proof (verbatim) from the first case where $X = \emptyset$ shows that $g_0 \in W^{n}$, where $n$ is a constant that depends only on the dimension of $M$.   
By Lemma \ref{lem:expanding}, we can find $h_1$ and $h_2$ in $W^2$ such that $D \subset h_1(A)$ and $D' \subset h_2(A')$.  Thus, for $i = 1, 2$ we have  $f_i^{-1} g_i h_i \in W^8$, hence $g_i \in W^{10}$.   This shows that $g \in W^{20+n}$, which is what we needed to show.  

\end{proof}

\begin{acknowledge}
The author is partially supported by NSF grant  DMS-1844516 
and a Sloan fellowship. 
Thanks to N. Vlamis and C. Rosendal for comments, and to the participants of the 2019 AIM workshop on mapping class groups of infinite type surfaces for the encouragement to publish this note. 
\end{acknowledge}


\bibliographystyle{plain}

\bibliography{biblio}

\end{document}